\documentclass[12pt,leqno]{article}  
\usepackage{amsmath,amssymb,bbm,accents, latexsym}
\usepackage{amsthm,array,mathdots}
\usepackage{titlesec} 
\usepackage[all,cmtip]{xy}
\usepackage{blkarray,arydshln} 

\usepackage{tikz}

\usepackage{picins}

\usepackage{graphics}
\usetikzlibrary{patterns}
\usepackage{lmodern}
\titlespacing{\section}{0cm}{3.5pc}{1.5pc}

\makeatletter
\def\@citex[#1]#2{\if@filesw\immediate\write\@auxout{\string\citation{#2}}\fi
  \def\@citea{}\@cite{\@for\@citeb:=#2\do
    {\@citea\def\@citea{\@citesep}\@ifundefined
       {b@\@citeb}{{\bf ?}\@warning
       {Citation `\@citeb' on page \thepage \space undefined}}%
{\csname b@\@citeb\endcsname}}}{}}
\def\@citesep{; }
\makeatother

\newtheoremstyle{Zhu}{}{}{\itshape}{}{\bf}{}{.5em}{}
\theoremstyle{Zhu}
\newtheoremstyle{Zhu}{}{}{\itshape}{}{\bf}{}{.5em}{}
\theoremstyle{Zhu}
\newtheorem{theorem}{Theorem}[section]

\newtheorem{lemma}[theorem]{Lemma}
\newtheorem{coro}[theorem]{Corollary}
\newtheorem{prop}[theorem]{Proposition}

\newtheoremstyle{Kremark}{}{}{}{}{\bf}{}{.5em}{}
\theoremstyle{Kremark}

\newtheorem{defn}[theorem]{Definition}
\newtheorem{example}[theorem]{Example}
\newtheorem{other}{}

\allowdisplaybreaks[1]  

\title{ Depth and  Stanley depth of the path ideal associated to an $n$-cyclic graph}
\author{\begin{minipage}{0.8\textwidth}
 \hspace{4.5cm}   Guangjun Zhu \\[2mm] \normalsize
School of Mathematical Sciences,
Soochow University,
Suzhou, China \\
\end{minipage} }

\vspace{-6cm}

\date{}

\begin{document}

\maketitle
\footnote{Supported by the National Natural Science Foundation of
China (11271275) and by  Foundation of Jiangsu Overseas Research \& Training Program for
University Prominent Young \& Middle-aged Teachers and Presidents and
 by Foundation of the Priority Academic Program Development of Jiangsu Higher Education Institutions. }

\footnote{ E-mail: zhuguangjun@suda.edu.cn}

\vspace{-2cm}
\begin{abstract}
{\noindent\bf Abstract.}  We compute the depth and Stanley depth for the quotient ring
of the path ideal of length $3$  associated to a $n$-cyclic graph, given some precise formulas for
depth when $n\not\equiv 1\,(\mbox{mod}\ 4)$, tight bounds when $n\equiv 1\,(\mbox{mod}\ 4)$
and for Stanley depth when $n\equiv 0,3\,(\mbox{mod}\ 4)$, tight bounds when $n\equiv 1,2\,(\mbox{mod}\ 4)$. Also, we give some  formulas for
depth and Stanley depth of a quotient  of the path ideals of length $n-1$ and $n$.

\vspace{3mm}

\vspace{3mm}
{\noindent\bf Mathematics Subject Classification ($2010$): 13C15; 13P10; 13F20.}

\end{abstract}


\section{Introduction }

\hspace{3mm} Let $S=K[x_{1},\dots,x_{n}]$ be a
 polynomial ring  in $n$ variables over a field $K$
 and  $M$ a finitely
generated $\mathbb{Z}^{n}$-graded $S$-module. For a homogeneous element $u\in M$
and  a subset $Z\subseteq
 \{x_{1},\dots,x_{n}\}$,
 $uK[Z]$ denotes  the $K$-subspace of $M$ generated by  all the homogeneous elements of the form $uv$, where
$v$ is a monomial in $K[Z]$.
 The $\mathbb{Z}^{n}$-graded $K$-subspace $uK[Z]$ is said to be  a Stanley space of
 dimension $|Z|$  if it is a free $K[Z]$-module, where, as usual, $|Z|$ denotes the number of elements of $Z$.
 A Stanley decomposition of $M$ is a decomposition of $M$
  as a finite direct sum of $\mathbb{Z}^{n}$-graded $K$-vector spaces
  $$\mathcal{D}:  M=
\bigoplus\limits_{i=1}^{r} u_{i}K[Z_{i}]$$
 where each $u_{i}K[Z_{i}]$ is a Stanley space of $M$.
 The number $\mbox{sdepth}_{S}\,(\mathcal{D})=
\mbox{min}\{|Z_{i}| : i = 1,\dots,r\}$ is called the Stanley depth of decomposition $\mathcal{D}$ and the quantity
$$\mbox{sdepth}\,(M):= \mbox{max}\{\mbox{sdepth}\,(\mathcal{D}) \mid  \mathcal{D}\  \mbox{is a
Stanley decomposition of}\  M\}.$$
is called the Stanley depth of $M$. Stanley  \cite{S} conjectured that
$$\mbox{sdepth}\,(M)\geq \mbox{depth}\,(M)$$
for all $\mathbb{Z}^{n}$-graded $S$-modules $M$. This conjecture proves to be false, in general,
for $M=S/I$ and $M=J/I$, where $I\subset J\subset S$ are monomial ideals, see \cite{DGK}.

Herzog, Vl\u{a}doiu and Zheng  \cite{HVZ} introduced a method to
compute the Stanley depth of a factor of a monomial ideal which was later developed into an effective
algorithm by Rinaldo  \cite{R3} implemented in CoCoA \cite{C3}. However, it is difficult to compute this invariant, even in some very particular cases. For instance in \cite{BHK}
Bir\'o et al. proved that $\mbox{sdepth}\,(\frak{m})=\lceil \frac{n}{2}\rceil$
where $\frak{m}=(x_{1},\dots, x_{n})$ is the
graded maximal ideal of $S$ and  $\lceil \frac{n}{2}\rceil$
denote the smallest integer $\geq \frac{n}{2}$.
For a friendly introduction on Stanley depth we refer the reader to \cite{H}.

Let $I_{n,m}$ and $J_{n,m}$ be the paths ideals of length $m$ associated to the $n$-line, respectively $n$-cyclic, graph.
Cimpoea{s}   \cite{C1} proved that $\mbox{depth}\,(S/J_{n,2})=\lceil \frac{n-1}{3}\rceil$ and   when
 $n\equiv 0\,(\mbox{mod}\ 3)$ or $n\equiv 2\,(\mbox{mod}\ 3)$, $\mbox{sdepth}\,(S/J_{n,2})=\lceil \frac{n-1}{3}\rceil$
and  when $n\equiv 1\,(\mbox{mod}\ 3)$,  $\lceil \frac{n-1}{3}\rceil\leq \mbox{sdepth}\,(S/J_{n,2})\leq \lceil \frac{n}{3}\rceil$.
In \cite{C2}, he  also showed that
  $\mbox{sdepth}\,(S/I_{n,m})=\mbox{depth}\,(S/I_{n,m})=n+1-\lfloor \frac{n+1}{m+1}\rfloor-\lceil \frac{n+1}{m+1}\rceil,$
where $\lfloor \frac{n+1}{m+1}\rfloor$
denote the biggest integer $\leq \frac{n+1}{m+1}$. Using similar techniques, we prove that
$\mbox{sdepth}\,(S/J_{n,3})=n-\lfloor \frac{n}{4}\rfloor-\lceil \frac{n}{4}\rceil$ for $n\equiv 0\,(\mbox{mod}\ 4)$ or $n\equiv 3\,(\mbox{mod}\ 4)$ and
$n-\lfloor \frac{n}{4}\rfloor-\lceil \frac{n}{4}\rceil\leq \mbox{sdepth}\,(S/J_{n,3})\leq n+1-\lfloor \frac{n}{4}\rfloor-\lceil \frac{n}{4}\rceil$ for $n\equiv 1\,(\mbox{mod}\ 4)$ or $n\equiv 2\,(\mbox{mod}\ 4)$.
Also, we prove that
 $\mbox{depth}\,(S/J_{n,3})=n-\lfloor \frac{n}{4}\rfloor-\lceil \frac{n}{4}\rceil$ for $n\not\equiv 1\,(\mbox{mod}\ 4)$ and  $n-\lfloor \frac{n}{4}\rfloor-\lceil \frac{n}{4}\rceil\leq \mbox{depth}\,(S/J_{n,3})\leq n+1-\lfloor \frac{n}{4}\rfloor-\lceil \frac{n}{4}\rceil$ for $n\equiv 1\,(\mbox{mod}\ 4)$.
In Proposition \ref{prop1},  we prove that
 $\mbox{sdepth}\,(J_{n,3}/I_{n,3})=n+1-\lfloor \frac{n}{4}\rfloor-\lceil \frac{n}{4}\rceil$ for all $n\geq 4$.
In the third section, we prove that $\mbox{sdepth}\,(\frac{S}{J_{n,n-1}})=\mbox{depth}\,(\frac{S}{J_{n,n-1}})=n-2$ and
$n-3\leq \mbox{sdepth}\,(\frac{S}{J_{n,n-2}}),\mbox{depth}\,(\frac{S}{J_{n,n-2}})\leq n-2$.

\section{Depth and Stanley depth of quotient of the path ideal with length $3$}

In this section, we will give some formulas for depth  and  stanley depth of quotient of the path ideals of length $3$.
 We first recall some definitions about graphs and their path ideals in order to make this paper
self-contained. However, for more details on the notions, we refer the reader to \cite{V2,Z1,Z2}.

\begin{defn} \label{def1} Let $G=(V,E)$ be a graph with vertex set  $V=\{x_{1},\dots,x_{n}\}$ and edge set $E$. Then
$G=(V,E)$  is called an $n$-line graph, denoted by  $L_{n}$, if its  edge set is given by $E=\{x_{i}x_{i+1}\mid 1\leq i\leq n-1\}$.  Similarly, if   $n\geq 3$,
then $G=(V,E)$  is called an $n$-cyclic graph, denoted by  $C_{n}$, if its  edge set  is given by $E=\{x_{i}x_{i+1}\mid 1\leq i\leq n-1\}\cup\{x_{n}x_{1}\}$.
\end{defn}

\begin{defn} \label{def2} Let $G=(V,E)$ be a graph  with vertex set  $V=\{x_{1},\dots,x_{n}\}$.
A path of length $m$ in $G$ is an alternating sequence of
vertices and edges $w\!=\!\{x_{i}, e_{i},x_{i+1},\dots, x_{i+m-2},\\
 e_{i+m-2},x_{i+m-1}\}$, where $e_{j}=x_{j}x_{j+1}$ is the edge joining $x_{j}$ and $x_{j+1}$.
 A path of length $m$ may also be denoted $\{x_{i},\dots,x_{i+m-1}\}$, the edges being evident
from the context.
\end{defn}

\begin{defn} \label{def3}Let $G=(V,E)$ be a graph with vertex set  $V=\{x_{1},\dots,x_{n}\}$.  Then the path ideal of  length $m$  associated to $G$ is the
squarefree monomial ideal $I=(x_{i}\cdots x_{i+m-1}\mid \{x_{i},\dots,x_{i+m-1}\}\, \mbox{is a path of length}\,\,  m\,\,
  \mbox{in}\  G)$ of $S$.
\end{defn}

In this paper, we set $n\geq 3$ and consider the $n$-line graph  $L_{n}$ and $n$-cyclic graph  $C_{n}$, their paths ideals  of length $m$ are denoted by $I_{m,n}$ and  $J_{m,n}$
respectively. Thus we obtain that $$I_{m,n}=(x_{i}\cdots x_{i+m-1}\mid 1\leq i\leq n-m+1),$$ and  $$J_{m,n}=I_{m,n}+(x_{n-m}\cdots x_{n}x_{1}, x_{n-m+1}\cdots x_{n}x_{1}x_{2},\dots,x_{n}x_{1}\cdots x_{m-1}).$$

\begin{defn} \label{def4} Let $(S,\frak{m})$ be a local ring (or a Noetherian graded ring with $(S_{0},\frak{m}_{0})$ local),
 $M$ a finite generated $S$-module with the property that $\frak{m}M\subsetneq M$ ( or  a finite generated graded $S$-module with the property that $(\frak{m}_{0}\oplus \bigoplus\limits_{i=1}^{\infty}S_{i})M\subsetneq M$).
 Then the depth of $M$,  is defined as
$$\mbox{depth}\,(M)=\mbox{min}\{i \mid \mbox{Ext}\,^{i}(S/\frak{m},M)\neq 0\}$$
(or
$\mbox{depth}\,(M)=\mbox{min}\{i \mid \mbox{Ext}\,^{i}(S/(\frak{m}_{0}\oplus \bigoplus\limits_{i=1}^{\infty}S_{i}),M)\neq 0\}$).
\end{defn}

\vspace{3mm}We recall the well known Depth Lemma, see for instance \cite[\, Lemma 1.3.9\,]{V2}  or \cite[\, Lemma 3.1.4\,]{V1}.

\begin{lemma}
\label{lem1}  (Depth Lemma) Let $0\rightarrow L\rightarrow M\rightarrow N \rightarrow 0$ be a short exact sequence of
modules over a local ring $S$, or a Noetherian graded ring with $S_{0}$ local, then
\begin{itemize}
 \item[(i)] $\mbox{depth}\,(M)\geq \mbox{min}\{\mbox{depth}\,(L), \mbox{depth}\,(N)\}$;
\item[(ii)] $\mbox{depth}\,(L)\geq \mbox{min}\{\mbox{depth}\,(M), \mbox{depth}\,(N)+1\}$;
\item[(iii)] $\mbox{depth}\,(N)\geq \mbox{min}\{\mbox{depth}\,(L)-1, \mbox{depth}\,(M)\}$.
\end{itemize}
\end{lemma}

\vspace{3mm}The most of the statements of the Depth Lemma are wrong if we replace depth by stanley depth. Some counter examples are given in
\cite[\, Example 2.5 and   2.6\,]{R2}.
Rauf \cite{R2} proved the analog of Lemma \ref{lem1} (i) for stanley depth.

\begin{lemma}
\label{lem2} Let $0\rightarrow L\rightarrow M\rightarrow N \rightarrow 0$ be a short exact sequence of finitely generated $\mathbb{Z}^{n}$-graded $S$-modules. Then
$$\mbox{sdepth}\,(M)\geq \mbox{min}\{\mbox{sdepth}\,(L), \mbox{sdepth}\,(N)\}.$$
\end{lemma}

\vspace{3mm}In  \cite{C1},  Cimpoea{s} computed depth and Stanley depth  for $S/J_{n,2}$.
\begin{lemma}
\label{lem3} \begin{itemize}
\item[(1)] $\mbox{depth}\,(S/J_{n,2})=\lceil \frac{n-1}{3}\rceil$;
\item[(2)] $\mbox{sdepth}\,(S/J_{n,2})=\lceil \frac{n-1}{3}\rceil$ for $n\equiv 0\,(\mbox{mod}\ 3)$  or $n\equiv 2\,(\mbox{mod}\ 3)$;
\item[(3)] $\lceil \frac{n-1}{3}\rceil \leq \mbox{sdepth}\,(S/J_{n,2})\leq \lceil \frac{n}{3}\rceil$ for $n\equiv 1\,(\mbox{mod}\ 3)$.
\end{itemize}
\end{lemma}

\vspace{3mm}In  \cite{C2},   Cimpoea{s} computed depth and Stanley depth  for $S/I_{n,m}$, which generalizes \cite[\, Lemma 2.8\,]{M}
 and \cite[\, Lemma 4\,]{S2}.
\begin{lemma}
\label{lem4} $\mbox{sdepth}\,(S/I_{n,m})=\mbox{depth}\,(S/I_{n,m})=n+1-\lfloor \frac{n+1}{m+1}\rfloor-\lceil \frac{n+1}{m+1}\rceil$. In particular,
$\mbox{sdepth}\,(S/I_{n,2})=\mbox{depth}\,(S/I_{n,2})=\lceil \frac{n}{3}\rceil$.\end{lemma}

\vspace{3mm}Using these lemmas, we are able to prove the main result of this section.
\begin{theorem}\label{Thm1}
\begin{itemize}
\item[(1)] $\mbox{depth}\,(S/J_{n,3})\geq n-\lfloor \frac{n}{4}\rfloor-\lceil \frac{n}{4}\rceil$;
\item[(2)] $\mbox{sdepth}\,(S/J_{n,3})\geq n-\lfloor \frac{n}{4}\rfloor-\lceil \frac{n}{4}\rceil$.
\end{itemize}
\end{theorem}
 \begin{proof}  These two results can be shown by similar arguments, so we only prove that
$\mbox{sdepth}\,(S/J_{n,3})\geq n-\lfloor \frac{n}{4}\rfloor-\lceil \frac{n}{4}\rceil$.
 Let $S_{t}$ be the polynomial ring
 in $t$ variables over a field.  The case $n=3$ is trivial.
The cases  $n=4$ and  $n=5$ follow from Examples \ref{exam1}  and \ref{exam2} respectively.

We may assume that $n\geq 6$.
 Let $k=\lfloor \frac{n}{4}\rfloor$ and $\varphi(n)=n-\lfloor \frac{n}{4}\rfloor-\lceil \frac{n}{4}\rceil$. One can easily see that
 $$\varphi(n)=\left\{\begin{array}{ll}
n-2k,&\mbox{if}\ n\equiv0\,(\mbox{mod}\ 4);\\
n-2k+1,&\mbox{otherwise}.\\
\end{array}\right.$$
We denote $u_{i}=x_{i}x_{i+1}x_{i+2}$ for $1\leq i\leq n-2$, $u_{n-1}=x_{n-1}x_{n}x_{1}$  and  $u_{n}=x_{n}x_{1}x_{2}$. Set  $L_{0}=J_{n,3}$,
$L_{1}=(L_{0}:x_{n})$ and  $U_{1}=(L_{0},x_{n})$.
Notice  that  $L_{0}=(u_{1},\dots,u_{n})$, $L_{1}=(u_{2},\dots,u_{n-4}, \frac{u_{n-2}}{x_{n}}, \frac{u_{n-1}}{x_{n}}, \frac{u_{n}}{x_{n}})$ and
$U_{1}=(u_{1},\dots,u_{n-3},x_{n})$. Since
$S/U_{1}\simeq S_{n-1}/I_{n-1,3}$, we obtain  $\mbox{sdepth}\,(S/U_{1})=\varphi(n)$ by Lemma \ref{lem4}.

We set  $L_{j+1}=(L_{j}:x_{4j})$ and $U_{j+1}=(L_{j},x_{4j})$ where $1\leq j\leq k-3$. One can easily check that:
$$L_{j+1}=(\frac{u_{2}}{x_{4}},\frac{u_{3}}{x_{4}},\frac{u_{4}}{x_{4}},\frac{u_{6}}{x_{8}},\dots,\frac{u_{4(j-1)}}{x_{4(j-1)}},\frac{u_{4j-2}}{x_{4j}},\dots,\frac{u_{4j}}{x_{4j}},
u_{4j+2},\dots,u_{n-4}, \frac{u_{n-2}}{x_{n}}, \frac{u_{n-1}}{x_{n}}, \frac{u_{n}}{x_{n}}),$$
 and
 $$U_{j+1}=(\frac{u_{2}}{x_{4}},\frac{u_{3}}{x_{4}},\frac{u_{4}}{x_{4}},\frac{u_{6}}{x_{8}},\dots,\frac{u_{4(j-1)-2}}{x_{4(j-1)}},\dots,\frac{u_{4(j-1)}}{x_{4(j-1)}}, x_{4j},
u_{4j+1},\dots,u_{n-4}, \frac{u_{n-2}}{x_{n}}, \frac{u_{n-1}}{x_{n}}, \frac{u_{n}}{x_{n}}),$$
where $x_{0}=1$ and $u_{j}=0$ for $j\leq 0$.

We consider the following three cases:

(1). If  $n=4k$ or  $n=4k-1$,  we denote $L_{j+1}=(L_{j}:x_{4j})$ and $U_{j+1}=(U_{j},x_{4j})$ for $j=k-2, k-1$. We conclude
 $L_{k}\simeq J_{n-k,2}S$, $U_{k}=(x_{4(k-1)},V_{k})$ where $V_{k}=(\frac{u_{2}}{x_{4}}, \frac{u_{3}}{x_{4}},\frac{u_{4}}{x_{4}}, \frac{u_{6}}{x_{8}},\dots,\frac{u_{4(k-2)}}{x_{4(k-2)}}, \frac{u_{n-2}}{x_{n}},\frac{u_{n-1}}{x_{n}},\frac{u_{n}}{x_{n}})$.
 Note that $$V_{k}\simeq \left\{\begin{array}{ll}
I_{n-k-2,2},&\mbox{if}\ n=4k;\\
I_{n-k-1,2},&\mbox{if}\ n=4k-1.\\
\end{array}\right.
$$
Thus, by Lemmas
\ref{lem3}, \ref{lem4} and \cite[Lemma  3.6]{HVZ}, it follows that
$$\mbox{sdepth}\,(S/L_{k})=k+\mbox{sdepth}\,(S_{n-k}/J_{n-k,2})=k+k=\varphi(n),$$
and
$$\mbox{sdepth}\,(\frac{S}{U_{k}})= \left\{\begin{array}{lr}
(k+1)+\mbox{sdepth}\,(\frac{S_{n-k-2}}{I_{n-k-2,2}})=(k+1)+k=1+\varphi(n),\hspace{-5mm}&\mbox{if}\ n=4k;\\
k+\mbox{sdepth}\,(\frac{S_{n-k-1}}{I_{n-k-1,2}})=k+k=\varphi(n),&\mbox{if}\ n=4k-1.\\
\end{array}\right.$$

 (2). If $n=4k-2$, we denote $L_{k-1}=(L_{k-2}:x_{4(k-2)})$,  $U_{k-1}=(L_{k-2},x_{4(k-2)})$,
$L_{k}=(L_{k-1}:x_{4(k-1)-1})$ and  $U_{k}=(U_{k-1},x_{4(k-1)-1})$. We have  $L_{k}\simeq J_{n-k,2}S$, $U_{k}=(x_{4(k-1)-1},V_{k})$  where $V_{k}=(\frac{u_{2}}{x_{4}}, \frac{u_{3}}{x_{4}},\frac{u_{4}}{x_{4}}, \frac{u_{6}}{x_{8}},\dots,\frac{u_{4(k-2)}}{x_{4(k-2)}}, \frac{u_{n-2}}{x_{n}},\frac{u_{n-1}}{x_{n}},\frac{u_{n}}{x_{n}})\simeq I_{n-k,2}S$.
Thus, by Lemma \ref{lem4} and \cite[Lemma  3.6]{HVZ}, we obtain
$$\mbox{sdepth}\,(S/U_{k})=(k-1)+\mbox{sdepth}\,(S_{n-k}/I_{n-k,2})=(k-1)+k=\varphi(n).$$
Applying  Lemma \ref{lem3} and \cite[Lemma  3.6]{HVZ}, we get
$$\mbox{sdepth}\,(S/L_{k})=k+\mbox{sdepth}\,(S_{n-k}/J_{n-k,2})\geq k+(k-1)=\varphi(n),$$
and
$$\mbox{sdepth}\,(S/L_{k})=k+\mbox{sdepth}\,(S_{n-k}/J_{n-k,2})\leq k+k=1+\varphi(n).$$

(3). If $n=4k-3$, we denote $L_{k-1}=(L_{k-2}:x_{4(k-2)-1})$,  $U_{k-1}=(L_{k-2},x_{4(k-2)-1})$,
$L_{k}=(L_{k-1}:x_{4(k-1)-2})$ and  $U_{k}=(U_{k-1},x_{4(k-1)-2})$.  We have   $L_{k}\simeq J_{n-k,2}S$, \\
$U_{k}=(x_{4(k-1)-2},V_{k})$  where $V_{k}=(\frac{u_{2}}{x_{4}}, \frac{u_{3}}{x_{4}},\frac{u_{4}}{x_{4}}, \frac{u_{6}}{x_{8}},\dots,\frac{u_{4(k-3)}}{x_{4(k-3)}},\frac{u_{4(k-2)-2}}{x_{4(k-2)-1}},\frac{u_{4(k-2)-1}}{x_{4(k-2)-1}}, \frac{u_{n-2}}{x_{n}},\\
\frac{u_{n-1}}{x_{n}}, \frac{u_{n}}{x_{n}})\simeq I_{n-k,2}S$.
Therefore, by Lemmas
\ref{lem3}, \ref{lem4} and \cite[Lemma  3.6]{HVZ}, we have
$$\mbox{sdepth}\,(S/L_{k})=k+\mbox{sdepth}\,(S_{n-k}/J_{n-k,2})= k+(k-1)=1+\varphi(n),$$
and
$$\mbox{sdepth}\,(S/U_{k})=(k-1)+\mbox{sdepth}\,(S_{n-k}/I_{n-k,2})=(k-1)+(k-1)=\varphi(n).$$
This shows that $\varphi(n)\leq \mbox{sdepth}\,(S/L_{k})\leq 1+\varphi(n)$ and $\mbox{sdepth}\,(S/U_{k})\geq \varphi(n) \hspace{0.5cm}(\ast)$.

 Consider the  following short exact sequences

$$\begin{array}{c}
0 \longrightarrow  \frac{S}{L_{1}}  \longrightarrow \frac{S}{L_{0}}  \longrightarrow    \frac{S}{U_{1}} \longrightarrow  0 \\

 \\
0 \longrightarrow  \frac{S}{L_{2}}  \longrightarrow \frac{S}{L_{1}}  \longrightarrow    \frac{S}{U_{2}} \longrightarrow  0 \\

 \\
 \vdots \hspace{15mm} \vdots \hspace{15mm}\vdots  \\

 \\
 0 \rightarrow  \frac{S}{L_{k-1}}  \rightarrow \frac{S}{L_{k-2}}  \rightarrow    \frac{S}{U_{k-1}} \rightarrow  0 \\

 \\
 0 \longrightarrow  \frac{S}{L_{k}}  \longrightarrow \frac{S}{L_{k-1}}  \longrightarrow    \frac{S}{U_{k}} \longrightarrow  0 \\
\end{array}$$

By Lemma \ref{lem2} and $(\ast)$, we have
\begin{eqnarray*}\mbox{sdepth}\,(\frac{S}{J_{n,3}})&=&\mbox{sdepth}\,(\frac{S}{L_{0}})\geq \mbox{min}\{\mbox{sdepth}\,(\frac{S}{L_{1}}), \mbox{sdepth}\,(\frac{S}{U_{1}})\}\\
&\geq &\mbox{min}\{\mbox{sdepth}\,(\frac{S}{L_{2}}),\mbox{sdepth}\,(\frac{S}{U_{2}}), \mbox{sdepth}\,(\frac{S}{U_{1}})\}\\
&\geq &\cdots\\
\end{eqnarray*}
\begin{eqnarray*}&\geq &\mbox{min}\{\mbox{sdepth}\,(\frac{S}{L_{k}}),\mbox{sdepth}\,(\frac{S}{U_{k}}),\mbox{sdepth}\,(\frac{S}{U_{k-1}}),\dots, \mbox{sdepth}\,(\frac{S}{U_{1}})\}\\
&\geq &\mbox{min}\{\varphi(n),\mbox{sdepth}\,(\frac{S}{U_{k-1}}),\dots, \mbox{sdepth}\,(\frac{S}{U_{2}}),\mbox{sdepth}\,(\frac{S}{U_{1}})\}
\end{eqnarray*}
To show $\mbox{sdepth}\,(\frac{S}{J_{n,3}})\geq \varphi(n)$ it is enough to prove the claim below.

Claim: $\mbox{sdepth}\,(S/U_{j+1})\geq \varphi(n)$ for all $1\leq j\leq k-2$.

For any $1\leq j\leq k-3$, we set  $V_{j+1}=(\frac{u_{n-2}}{x_{n}},\frac{u_{n-1}}{x_{n}},\frac{u_{n}}{x_{n}},\frac{u_{2}}{x_{4}},\dots, \frac{u_{4(j-1)-2}}{x_{4(j-1)}},\dots, \frac{u_{4(j-1)}}{x_{4(j-1)}})$ and $W_{j+1}=(u_{4j+1}, \dots, u_{n-4})$  where $x_{0}=1$ and $u_{j}=0$ for $j\leq 0$.  We have   $\frac{S}{U_{j+1}}\simeq \frac{S/V_{j+1}\oplus S/W_{j+1}}{x_{4j}(S/V_{j+1}\oplus S/W_{j+1})}$.
Since  $x_{4j}$ is regular on $S/V_{j+1}\oplus S/W_{j+1}$, by  \cite[Theorem  1.1]{R1} and \cite[Theorem  1.3]{C4}, we have $$\mbox{sdepth}\,(\frac{S}{U_{j+1}})=\mbox{sdepth}\,(\frac{S}{V_{j+1}}\oplus \frac{S}{W_{j+1}})-1\geq \mbox{sdepth}\,(\frac{S}{V_{j+1}})+\mbox{sdepth}\,(\frac{S}{W_{j+1}})-n-1.$$

On the other hand, $V_{j+1}\simeq I_{3j+1,2}S$, $W_{j+1}\simeq I_{n-4(j+1)+2,3}S$. Thus, by Lemma \ref{lem4}, we have
$$\mbox{sdepth}\,(S/V_{j+1})=[n-(3j+1)]+\lceil \frac{3j+1}{3}\rceil=n-2j$$ and
 \begin{eqnarray*}\mbox{sdepth}\,(S/W_{j+1})&=&[4(j+1)-2]+[n-4(j+1)+3]-\lfloor \frac{n-4(j+1)+3}{4}\rfloor\\
 &-&\lceil \frac{n-4(j+1)+3}{4}\rceil\\
 &=&n+1-\lfloor \frac{n-4j-1}{4}\rfloor-\lceil \frac{n-4j-1}{4}\rceil\\
 &=&n+1+2j-\lfloor \frac{n-1}{4}\rfloor-\lceil \frac{n-1}{4}\rceil.
 \end{eqnarray*}

 By some simple computations, we conclude that   \begin{eqnarray*}\mbox{sdepth}\,(S/U_{j+1})&\geq & (n-2j)+(n+1+2j)-\lfloor \frac{n-1}{4}\rfloor-\lceil \frac{n-1}{4}\rceil-n-1\\
 &=&n-\lfloor \frac{n-1}{4}\rfloor-\lceil \frac{n-1}{4}\rceil\geq \varphi(n).
  \end{eqnarray*}

  If $n\neq 4k-3$, we have    $V_{k-1}=(\frac{u_{n-2}}{x_{n}},\frac{u_{n-1}}{x_{n}},\frac{u_{n}}{x_{n}},\frac{u_{2}}{x_{4}},\dots, \frac{u_{4(k-3)-2}}{x_{4(k-3)}},\dots, \frac{u_{4(k-3)}}{x_{4(k-3)}})$ and $W_{k-1}=(u_{4(k-2)+1}, \dots, u_{n-4})$. It follows from  similar arguments as above.

   If $n=4k-3$, we have  $V_{k-1}=(\frac{u_{n-2}}{x_{n}},\frac{u_{n-1}}{x_{n}},\frac{u_{n}}{x_{n}},\frac{u_{2}}{x_{4}},\dots, \frac{u_{4(k-3)-2}}{x_{4(k-3)}},\dots, \frac{u_{4(k-3)}}{x_{4(k-3)}})$
and $W_{k-1}=(u_{4(k-2)}, \dots, u_{n-4})$.  Note that
 $V_{k-1}\simeq I_{3(k-2)+1,2}S$ and $W_{k-1}\simeq I_{n-4(k-1)+3,3}S$.
  Thus, by Lemma \ref{lem4}, we obtain
$$\mbox{sdepth}\,(S/V_{k-1})=(n-(3(k-2)+1))+\lceil \frac{3(k-2)+1}{3}\rceil=n-2k+4=2k+1$$ and
 \begin{eqnarray*}\mbox{sdepth}\,(S/W_{k-1})&=&(4(k-1)-3)+(n-4(k-1)+4)-\lfloor \frac{n-4(k-1)+4}{4}\rfloor\\
 &-&\lceil \frac{n-4(k-1)+4}{4}\rceil\\
  &=& n+1-\lfloor \frac{n-4k+8}{4}\rfloor-\lceil \frac{n-4k+8}{4}\rceil\\
 &=& n+1-\lfloor \frac{5}{4}\rfloor-\lceil \frac{5}{4}\rceil=n-2. \end{eqnarray*}

One can easily see that $\frac{S}{U_{k-1}}\simeq \frac{S/V_{k-1}\oplus S/W_{k-1}}{x_{4(k-2)-1}(S/V_{k-1}\oplus S/W_{k-1})}$.
Since  $x_{x_{4(k-2)-1}}$ is regular on $S/V_{k-1}\oplus S/W_{k-1}$, by  \cite[Theorem  1.1]{R1} and \cite[Theorem  1.3]{C4}, we have
\begin{eqnarray*}
\mbox{sdepth}\,(\frac{S}{U_{k-1}})&=&\mbox{sdepth}\,(\frac{S}{V_{k-1}}\oplus \frac{S}{W_{k-1}})-1\\
&\geq& \mbox{sdepth}\,(\frac{S}{V_{k-1}})+\mbox{sdepth}\,(\frac{S}{W_{k-1}})-n-1
\\&=&2k+1+(n-2)-n-1\\
&=&2k-2=\varphi(n).
 \end{eqnarray*}
This completes the proof.
\end{proof}

\vspace{3mm}\begin{example} \label{exam1} Let $J_{4,3}=(x_{1}x_{2}x_{3},x_{2}x_{3}x_{4},x_{3}x_{4}x_{1},x_{4}x_{1}x_{2})\subset S=K[x_{1},\dots,x_{4}]$. Note
that $4-\lfloor \frac{4}{4}\rfloor-\lceil \frac{4}{4}\rceil=2$. Set $L_{1}=(J_{4,3}:x_{4})$
and $U_{1}=(J_{4,3},x_{4})$. Since $L_{1}=(x_{1}x_{2},x_{2}x_{3},x_{3}x_{1})=J_{3,2}S$ and  $U_{1}=(x_{1}x_{2}x_{3},x_{4})$. Thus $S/U_{1}=K[x_{1},x_{2},x_{3}]/(x_{1}x_{2}x_{3})$.
By Lemmas \ref{lem3},  \ref{lem4} and \cite[Lemma  3.6]{HVZ}, we have
$\mbox{sdepth}\,(S/L_{1})=\mbox{depth}\,(S/L_{1})=1+\lceil \frac{3-1}{3}\rceil=2$  and $\mbox{sdepth}\,(S/U_{1})=2$. Applying Lemma \ref{lem2} on the  short exact sequence
$$ 0\longrightarrow S/L_{1}\longrightarrow S/J_{4,3}\longrightarrow S/U_{1}\longrightarrow 0,$$
we obtain   $\mbox{depth}\,(\frac{S}{J_{4,3}})=2$ and $\mbox{sdepth}\,(\frac{S}{J_{4,3}})\geq 2$. By \cite[Proposition 2.7]{C4}, it follows that
$\mbox{sdepth}\,(\frac{S}{J_{4,3}})\leq \mbox{sdepth}\,(S/L_{1})=2$. Thus $\mbox{sdepth}\,(S/J_{4,3})=2$.
\end{example}

\begin{example} \label{exam2} Let $J_{5,3}=(x_{1}x_{2}x_{3},x_{2}x_{3}x_{4},x_{3}x_{4}x_{5},x_{4}x_{5}x_{1},x_{5}x_{1}x_{2})\subset S=K[x_{1},\dots,x_{5}]$. Note
that $5-\lfloor \frac{5}{4}\rfloor-\lceil \frac{5}{4}\rceil=2$. Set $L_{1}=(J_{5,3}:x_{5})$
and $U_{1}=(J_{5,3},x_{5})$. Since $L_{1}=(x_{3}x_{4},x_{4}x_{1},x_{1}x_{2})\simeq I_{4,2}S$ and  $U_{1}=(x_{1}x_{2}x_{3},x_{2}x_{3}x_{4},x_{5})$. Thus
 $S/U_{1}= S_{4}/I_{4,3}$, by Lemma \ref{lem4} and  \cite[Lemma  3.6]{HVZ}, we have
$\mbox{sdepth}\,(S/L_{1})=\mbox{depth}\,(S/L_{1})=1+\lceil \frac{4}{3}\rceil=3$  and $\mbox{sdepth}\,(S/U_{1})=5-\lfloor \frac{5}{4}\rfloor-\lceil \frac{5}{4}\rceil=2$. Using Lemmas \ref{lem1} and \ref{lem2} on the  short exact sequence
$$ 0\longrightarrow S/L_{1}\longrightarrow S/J_{5,3}\longrightarrow S/U_{1}\longrightarrow 0,$$
we obtain $\mbox{depth}\,(S/J_{5,3})\geq 2$ and $\mbox{sdepth}\,(S/J_{5,3})\geq 2$.
\end{example}

 \vspace{3mm}As a consequence of Theorem \ref{Thm1}, one has the following results.
\begin{coro} \label{cor1}
\begin{itemize}
\item[(1)] $\mbox{sdepth}\,(S/J_{n,3})\leq n+1-\lfloor \frac{n}{4}\rfloor-\lceil \frac{n}{4}\rceil$ for $n\equiv 1\,(\mbox{mod}\ 4)$  or $n\equiv 2\,(\mbox{mod}\ 4)$;\\
\item[(2)]
$\mbox{sdepth}\,(S/J_{n,3})=n-\lfloor \frac{n}{4}\rfloor-\lceil \frac{n}{4}\rceil$ for $n\equiv 0\,(\mbox{mod}\ 4)$ or $n\equiv 3\,(\mbox{mod}\ 4)$.
\end{itemize}
\end{coro}
  \begin{proof} Set $\varphi(n)=n-\lfloor \frac{n}{4}\rfloor-\lceil \frac{n}{4}\rceil$. From the proof of  Theorem \ref{Thm1}, we
see that $\mbox{sdepth}\,(S/L_{k})\le 1+\varphi(n)$ for  $n\equiv 1\,(\mbox{mod}\ 4)$  or $n\equiv 2\,(\mbox{mod}\ 4)$ and otherwise, $\mbox{sdepth}\,(S/L_{k})=\varphi(n)$. These are a direct consequence of  \cite[Proposition 2.7]{C4}.
  \end{proof}

\begin{coro} \label{cor2}
 \begin{itemize}
\item[(1)] $\mbox{depth}\,(S/J_{n,3})\leq n+1-\lfloor \frac{n}{4}\rfloor-\lceil \frac{n}{4}\rceil$ for $n\equiv 1\,(\mbox{mod}\ 4)$,\\
\item[(2)] $\mbox{depth}\,(S/J_{n,3})=n-\lfloor \frac{n}{4}\rfloor-\lceil \frac{n}{4}\rceil$ for $n\not\equiv 1\,(\mbox{mod}\ 4)$.
\end{itemize}
 \end{coro}
  \begin{proof} Set $\varphi(n)=n-\lfloor \frac{n}{4}\rfloor-\lceil \frac{n}{4}\rceil$. Replacing stanley depth
by depth in the proof of  Theorem \ref{Thm1}, we
see that $\mbox{depth}\,(S/L_{k})=1+\varphi(n)$ for  $n\equiv 1\,(\mbox{mod}\ 4)$ and otherwise, $\mbox{depth}\,(S/L_{k})=\varphi(n)$. These are a direct consequence of
 \cite[Corollary 1.3]{R2}.
  \end{proof}

\vspace{3mm}
 \begin{prop}
 \label{prop1}  $\mbox{sdepth}\,(J_{n,3}/I_{n,3})=n+1-\lfloor \frac{n}{4}\rfloor-\lceil \frac{n}{4}\rceil$ for all $n\geq 4$.
 \end{prop}
\begin{proof} One can easily check that $J_{4,3}/I_{4,3}\simeq x_{1}x_{3}x_{4}K[x_{1},x_{3},x_{4}]\oplus x_{1}x_{2}x_{4}K[x_{1},x_{2},x_{4}]$.
Thus, $\mbox{sdepth}\,(J_{4,3}/I_{4,3})=3$, as required. Similarly,
for $n=5$, we have $J_{5,3}/I_{5,3}\simeq x_{1}x_{4}x_{5}K[x_{1},x_{4},x_{5}]\oplus x_{1}x_{2}x_{5}K[x_{1},x_{2},x_{5}]\oplus x_{1}x_{2}x_{4}x_{5}K[x_{1},x_{2},x_{4},x_{5}]$,
for $n=6$, we have  $J_{6,3}/I_{6,3}\simeq x_{1}x_{5}x_{6}K[x_{1},x_{3},x_{5},x_{6}]\oplus x_{1}x_{2}x_{6}K[x_{1},x_{2},x_{4},x_{6}]\oplus x_{1}x_{2}x_{5}x_{6}K[x_{1},x_{2},x_{5},x_{6}]$ and for $n=7$, we get  $J_{7,3}/I_{7,3}\simeq x_{1}x_{6}x_{7}K[x_{1},x_{3},x_{4},x_{6},x_{7}]\oplus x_{1}x_{2}x_{7}K[x_{1},x_{2},x_{4},x_{5},x_{7}]\oplus x_{1}x_{2}x_{6}x_{7}K[x_{1},x_{2},x_{4},x_{6},x_{7}]$.

Now, assume $n\geq 8$, and let $u\in J_{n,3}$ be a monomial such that $u\notin I_{n,3}$. It follows that $u=x_{1}x_{n-1}x_{n}v_{1}$ or $u=x_{1}x_{2}x_{n}v_{2}$,
 with $v_{1}\in K[x_{1},\dots,x_{n-3},x_{n-1},x_{n}]$ and $v_{2}\in K[x_{1},x_{2},x_{4},\dots,x_{n}]$. We can write $v_{1}=x_{1}^{\alpha}x_{n-1}^{\beta}x_{n}^{\gamma}w$
with $w\in K[x_{2},\dots,x_{n-3}]$. Since $u\notin I_{n,3}$, it follows that $w\notin (x_{2}x_{3},x_{3}x_{4}x_{5},\dots,x_{n-5}x_{n-4}x_{n-3})$.  Similarly,
we can write $v_{2}=x_{1}^{\alpha}x_{2}^{\beta}x_{n}^{\gamma}w$
with $w\in K[x_{4},\dots,x_{n-1}]$. Since $u\notin I_{n,3}$, it follows that $w\notin (x_{4}x_{5}x_{6},\dots,x_{n-4}x_{n-3}x_{n-2},x_{n-2}x_{n-1})$.
Therefore, we have the $S$-module isomorphism:
\begin{eqnarray*} J_{n,3}/I_{n,3}&\simeq& x_{1}x_{2}x_{n}(\frac{K[x_{4},\dots,x_{n-2}]}{(x_{4}x_{5}x_{6},\dots,x_{n-4}x_{n-3}x_{n-2})})[x_{1},x_{2},x_{n}]\\
&\oplus & x_{1}x_{n-1}x_{n}(\frac{K[x_{3},\dots,x_{n-3}]}{(x_{3}x_{4}x_{5},\dots,x_{n-5}x_{n-4}x_{n-3})})[x_{1},x_{n-1},x_{n}]\\
&\oplus & x_{1}x_{2}x_{n-1}x_{n}(\frac{K[x_{4},\dots,x_{n-3}]}{(x_{4}x_{5}x_{6},\dots,x_{n-5}x_{n-4}x_{n-3})})[x_{1},x_{2},x_{n-1},x_{n}].
\end{eqnarray*}
 Therefore, by Lemma \ref{lem4} and \cite[Lemma  3.6]{HVZ}, we obtain
 \begin{eqnarray*}\hspace{-5mm}\mbox{sdepth}\,(\frac{J_{n,3}}{I_{n,3}})\!\!\!\!&=&\!\!\!\!
 \mbox{min}\,\{3\!+\!(n\!-\!4)\!-\!\lfloor\! \frac{\!n-4\!}{4}\!\rfloor\!-\!\lceil \frac{\!n-4\!}{4}\rceil,
 4\!+\!(n\!-\!5)\!-\!\lfloor\! \frac{\!n-5\!}{4}\!\rfloor-\lceil \frac{\!n-5\!}{4}\rceil\}\\
 &=& \mbox{min}\,\{n+1-\lfloor \frac{n}{4}\rfloor-\lceil \frac{n}{4}\rceil,n+1-\lfloor \frac{n-1}{4}\rfloor-\lceil \frac{n-1}{4}\rceil\}\\
 &=&n+1-\lfloor \frac{n}{4}\rfloor-\lceil \frac{n}{4}\rceil.\end{eqnarray*}
 \end{proof}

\section{Depth and Stanley depth of quotient of the path ideal of length $n-1$ or $n-2$}

In this section, we will give some formulas for depth and  stanley depth  of quotient of the path ideal of length $n-1$ or $n-2$.

\begin{prop}
 \label{prop2} $\mbox{sdepth}\,(S/J_{n,n-1})=\mbox{depth}\,(S/J_{n,n-1})=n-2$.
 \end{prop}
\begin{proof}  We apply induction on $n$.
 The case $n=3$ follows from Lemma \ref{lem3}.
Assume now that $n\geq 4$.
Since $J_{n,n-1}=(\prod\limits_{i=1}^{n-1}x_{i},\prod\limits_{i=2}^{n}x_{i},(\prod\limits_{i=3}^{n}x_{i})x_{1}, \dots,(\prod\limits_{i=k}^{n}x_{i})(\prod\limits_{i=1}^{k-2}x_{i}),\dots,\\
x_{n}\prod\limits_{i=1}^{n-2}x_{i})$, we obtain  $(J_{n,n-1}:x_{n})=(\prod\limits_{i=1}^{n-2}x_{i},\prod\limits_{i=2}^{n-1}x_{i},(\prod\limits_{i=3}^{n-1}x_{i})x_{1}, \dots,(\prod\limits_{i=k}^{n-1}x_{i})(\prod\limits_{i=1}^{k-2}x_{i}),\dots,\\
x_{n-1}\prod\limits_{i=1}^{n-3}x_{i})=J_{n-1,n-2}S$,
$(J_{n,n-1},x_{n})=(\prod\limits_{i=1}^{n-1}x_{i}, x_{n})$. Hence we get  $S/(J_{n,n-1}:x_{n})=(S_{n-1}/J_{n-1,n-2})[x_{n}]$.  Using the induction hypothesis and \cite[Lemma  3.6]{HVZ}, we conclude
$$\mbox{sdepth}\,(S/(J_{n,n-1}:x_{n}))=1+\mbox{sdepth}\,(S_{n-1}/J_{n-1,n-2})=n-2,$$
 and
 $$\mbox{depth}\,(S/(J_{n,n-1}:x_{n}))=1+\mbox{depth}\,(S_{n-1}/J_{n-1,n-2})=n-2.$$
 On the other hand, we obtain $\mbox{sdepth}\,(S/(J_{n,n-1},x_{n}))=n-2$
by \cite[Theorem 1.1]{R1}.
By applying Lemmas \ref{lem1} and \ref{lem2} on  the   exact sequence
$$0\longrightarrow S/(J_{n,n-1}:x_{n})\stackrel{\cdot x_{n}}\longrightarrow S/J_{n,n-1}\longrightarrow S/(J_{n,n-1},x_{n})\longrightarrow 0,$$
 we obtain $\mbox{depth}\,(S/J_{n,n-1})\geq n-2$ and $\mbox{sdepth}\,(S/J_{n,n-1})\geq n-2$.
Therefore, it follows that $\mbox{sdepth}\,(S/J_{n,n-1})=n-2$ by  \cite[Proposition 2.7]{C4}.
\end{proof}

 \vspace{3mm}
\begin{prop}
 \label{prop3} \begin{itemize}
\item[(1)] $n-3\leq \mbox{sdepth}\,(S/J_{n,n-2})\leq n-2$,\\
\item[(2)] $n-3\leq \mbox{depth}\,(S/J_{n,n-2})\leq n-2$.
\end{itemize}
 \end{prop}
\begin{proof}  The case $n=3$ is trivial.
 The case $n=4$ follows from Lemma \ref{lem3}. We may assume that $n\geq 5$.
Set $L_{0}=J_{n,n-2}$, $L_{j}=(L_{j-1}:x_{n-j+1})$ and $U_{j}=(L_{j-1},x_{n-j+1})$ for all $1\leq j\leq n-4$.
We conclude that $L_{0}=(\prod\limits_{i=1}^{n-2}\!x_{i},\prod\limits_{i=2}^{n-1}\!x_{i},\prod\limits_{i=3}^{n}\!x_{i},(\prod\limits_{i=4}^{n}x_{i})x_{1}, \dots,\\
(\prod\limits_{i=k}^{n}\!x_{i})(\prod\limits_{i=1}^{k-3}x_{i}),\dots,x_{n}\prod\limits_{i=1}^{n-3}x_{i})$,
$L_{1}=(\prod\limits_{i=3}^{n-1} x_{i}, (\!\prod\limits_{i=4}^{n-1}\!x_{i})x_{1}, \dots,(\!\prod\limits_{i=k}^{n-1}\!x_{i})(\!\prod\limits_{i=1}^{k-3}\!x_{i}),\dots,
x_{n-1}\!\!\prod\limits_{i=1}^{n-4}x_{i},\\
\prod\limits_{i=1}^{n-3}\!x_{i})$,  and
$U_{1}=(\prod\limits_{i=1}^{n-2} x_{i},\prod\limits_{i=2}^{n-1} x_{i},x_{n})$. Since $S/U_{1}=S_{n-1}/I_{n-1,n-2}$,
we obtain  \\
$\mbox{sdepth}\,(S/U_{1})=\mbox{depth}\,(S/U_{1})=n-3$ by Lemma \ref{lem4}.
For any $1\leq j\leq n-4$,  by some simple computations, one can  see that $$L_{j}=(\prod\limits_{i=3}^{n-j} x_{i}, (\prod\limits_{i=4}^{n-j}x_{i})x_{1}, \dots,(\prod\limits_{i=k}^{n-j}x_{i})(\prod\limits_{i=1}^{k-3}x_{i}),\dots,x_{n-j}\prod\limits_{i=1}^{n-j-3}x_{i},\prod\limits_{i=1}^{n-j-2}x_{i}),$$
and $U_{j}=(U_{j-1},x_{n-j+1})=(\prod\limits_{i=1}^{n-j-1}\!x_{i},x_{n-j+1})$.
In particular, $L_{n-4}=(x_{3}x_{4},x_{4}x_{1},x_{1}x_{2})$ and $S/L_{n-4}\simeq (S_{4}/I_{4,2})[x_{5},\dots, x_{n}]$.
Therefore,  by Lemma \ref{lem4} and \cite[Lemma  3.6]{HVZ}, we get   $\mbox{sdepth}\,(S/L_{n-4})=\mbox{depth}\,(S/L_{n-4})=(n-4)+5-\lfloor \frac{5}{3}\rfloor-\lceil \frac{5}{3}\rceil=n-2$.
 On the other hand, we obtain  $\mbox{sdepth}\,(S/U_{j})=\mbox{depth}\,(S/U_{j})=n-2$ by \cite[Theorem 1.1]{R1}.
By applying Lemmas \ref{lem1} and \ref{lem2}  on  the   exact sequences
$$ 0\longrightarrow S/L_{j}\stackrel{\cdot x_{n-j+1}}\longrightarrow S/L_{j-1}\longrightarrow S/U_{j}\longrightarrow 0\hspace{5mm}\mbox{for}\
1\leq j\leq n-4,$$
 we  conclude   $\mbox{depth}\,(S/J_{n,n-2})\geq n-3$ and $\mbox{sdepth}\,(S/J_{n,n-2})\geq n-3$.

On the other hand, by \cite[Theorem 1.1]{R1} and  \cite[Proposition 2.7]{C4}, we have  $\mbox{depth}\,(S/J_{n,n-2})\leq \mbox{depth}\,(S/L_{n-4})$ and $\mbox{sdepth}\,(S/J_{n,n-2})\leq \mbox{sdepth}\,(S/L_{n-4})$.
This completes the proof.
\end{proof}

\end{document}